\documentclass[11pt,fleqn]{elsart} 
\usepackage[latin1]{inputenc}

\usepackage{graphicx}
\usepackage{psfrag}
\usepackage{amsmath}
\usepackage{multibib}
\usepackage{tabularx}
\usepackage{verbatim}
\usepackage{float}
\usepackage{amssymb}
\usepackage{latexsym}
\usepackage{pifont}
\usepackage{subfigure}

\newcommand{\RR}{\mathbb{R}}
\newcommand{\CC}{\mathbb{C}}
\newcommand{\NN}{\mathbb{N}}
\newcommand{\cnj}[1]{\overline{#1}} 
\newcommand{\ds}{\displaystyle}

\newcommand{\kron}{\otimes}

   {\noindent\begin{description}\item[\fbox{Problem #1}]}
   {\end{description}}

\newenvironment{proof}
   {\begin{pf*}{Proof.}}
   {\hfill $\Box$\end{pf*}}

   {\begin{exmp}[#1]}
   {\end{exmp}}

\newcommand{\mtxa}[2]{
\left[
\begin{array}{#1}
#2
\end{array}
\right]}

\begin{document}

\begin{frontmatter}

\title{Polynomial two-parameter eigenvalue problems and matrix
 pencil methods for stability of delay-differential equations}



\author[Elias]{Elias Jarlebring}\ead{e.jarlebring tu-bs.de},               
\thanks[Michielthanks]{The research of this author was supported in part by NSF grant DMS-0405387}
\author[Michiel]{Michiel~E.~Hochstenbach\thanksref{Michielthanks}}
\ead[url]{http://www.win.tue.nl/$\sim$hochsten/},    

\address[Elias]{Technische Universität Braunschweig, Institut Computational Mathematics, 38023 Braunschweig, Germany}             
\address[Michiel]{Department of Mathematics and Computing Science, Eindhoven University of Technology, PO Box 513, 5600 MB, The Netherlands.}


\begin{abstract}
Several recent methods used to analyze {\em asymptotic stability} of {\em
delay-differential equations} (DDEs) involve determining the eigenvalues of a
matrix, a matrix pencil or a {\em matrix polynomial} constructed by {\em Kronecker
products}. 
Despite some similarities between the different types of these so-called {\em
matrix pencil methods}, the general ideas used as well as the proofs  differ
considerably. Moreover, the available theory hardly reveals the relations
between the different methods.

In this work, a different derivation of various matrix pencil methods is presented
using a unifying framework of a new type of eigenvalue problem:
the {\em polynomial two-parameter eigenvalue problem}, of which the
{\em quadratic two-parameter eigenvalue problem} is a special case.
This framework makes it possible to establish relations between various seemingly
different methods and provides further insight in the theory of matrix pencil methods.

We also recognize a few new matrix pencil variants to determine
DDE stability.
Finally, the recognition of the new types of eigenvalue problem opens
a door to efficient computation of DDE stability.
\end{abstract} 

\end{frontmatter}

\section{Introduction}
Mathematical models consisting of {\em delay-differential equations (DDEs)}, in the
simplest form
\begin{equation}
  \dot{x}(t)=A_0x(t)+A_1x(t-\tau),
  \quad A_0, A_1\in\CC^{n\times n},
  \quad \tau\ge 0, \label{dde0}
\end{equation}
occur naturally in a wide variety of fields related to applied
mathematics, such as engineering, control theory, biology,
traffic modeling, neural networks, mechanics and electronic circuits.
For most applications, it is
desired that $x(t)\rightarrow 0$ as $t\rightarrow\infty$ for any bounded initial condition. This property is referred to as {\em
asymptotic stability}. If a DDE is not asymptotically 
stable, the linear model typically breaks down or the modeled system
has unwanted properties, such as oscillations or the energy content of
the system 
is unbounded in time causing the modeled physical object to break
 or at least turn inefficient. Clearly, asymptotic stability is important
 in practice and numerical and analytical tools to analyze asymptotic
 stability of DDEs is a popular topic of research. 
For instance, large parts of
 several monographs, which are standard references in the field of DDEs,  deal with
 stability of 
 DDEs, e.g., the books of  
Bellman \cite{Bellman:1963:DIFFERENTIAL},  
Niculescu \cite{Niculescu:2001:DELAY}, Michiels and Niculescu
\cite{Michiels:2007:STABILITYBOOK} and 
Gu, Kharitonov, and Chen \cite{Gu:2003:STABILITY}. See also the survey papers
\cite{Richard:2003:SUMMARY}, \cite{Gu:2003:SURVEY}, and
\cite{Kharitonov:1999:ROBUST}. This paper is concerned with asymptotic stability for
\eqref{dde0} as well as more general DDEs, e.g., DDEs with delays in the
derivative (called {\em neutral DDEs}) and DDEs with 
multiple delays. These more general DDEs will be considered later
in this paper. We mention that DDEs that are not of neutral type
are also called {\em retarded}.

Asymptotic stability is often described using the solutions of the
characteristic equation associated with \eqref{dde0}:
\[
   \det(-\lambda I+A_0+A_1e^{-\tau \lambda})=0,
\]
of which the solutions $\lambda$ are called eigenvalues; the set of
eigenvalues is called the {\em spectrum}.
The DDE \eqref{dde0} is asymptotically stable if the spectrum is
contained in the open left half plane
(see, e.g., \cite[Prop.~1.6]{Michiels:2007:STABILITYBOOK}).

We will also consider more general classes of DDEs in this paper.
For some of these DDEs, in particular neutral DDEs, it is not
sufficient that all eigenvalues have negative real parts to ensure
asymptotic stability. A neutral DDE is asymptotically stable if and only
if the supremum of the real part of the spectrum is negative
(see, e.g., \cite[Prop.~1.20]{Michiels:2007:STABILITYBOOK}).

Because of these properties, explicit conditions such that
there is a purely imaginary eigenvalue can be very useful in a stability
analysis. In this paper we will study explicit conditions
on the delay $\tau$ such that there is at least one purely imaginary
eigenvalue. In the literature, there are several 
approaches to characterize these values of $\tau$, sometimes called 
{\em critical delays}, {\em switching delays}, {\em crossing delays}, or
{\em kernel and offspring curves}; see
\cite[Remark~3.1]{Jarlebring:2008:THESIS} for some comments on terminology. 

One approach to determine critical delays is to consider 
the eigenvalues of certain matrices or matrix pencils
constructed by Kronecker products. Methods of this type are
presented by Chen, Gu, and Nett 
\cite{Chen:1995:DELAYMARGINS} (see also
\cite[Thm.~2.13]{Gu:2003:STABILITY}), 
Louisell \cite{Louisell:2001:IMAG},
Niculescu \cite{Niculescu:1998:HYPERBOLICITY}
(see also \cite[Proposotion~4.5]{Michiels:2007:STABILITYBOOK}),
and Fu, Niculescu, and Chen \cite{Fu:2006:NEUTRAL,Niculescu:2006:DDAE}.
The works \cite{Jarlebring:2008:JCAM} 
(see also \cite[Chapter~3]{Jarlebring:2008:THESIS}),
\cite{Jarlebring:2007:ECC} and \cite{Ergenc:2007:EXTENDED} also use 
a formulation of eigenvalue problems containing Kronecker products. Even
though these 
popular methods have some characteristics in 
common, the ideas used in the derivations differ. For
instance, Louisell \cite{Louisell:2001:IMAG} derives a result for
neutral DDEs by considering a linear ODE which is proven to share imaginary eigenvalues
with the DDE whereas Chen, Gu, and Nett \cite{Chen:1995:MAXIMALDELAY} and
several other authors depart from the characteristic equation and
exploit the fact that the eigenvalues of Kronecker products are
products of the eigenvalues of the individual factor matrices.

The goals of this paper are:
\begin{itemize}
\item[(a)] to introduce a new type of eigenvalue problem, 
the {\em polynomial two-parameter eigenvalue problem}, with the
{\em quadratic two-parameter eigenvalue problem} as important special case;
\item[(b)] to show the relevance of this problem to determine
critical delays for various types of DDEs;
\item[(c)] to provide alternative derivations of existing matrix pencil methods
using the context of polynomial two-parameter eigenvalue problems;
\item[(d)] to hereby provide a new unifying framework for
the determination of critical delays;
\item[(e)] and, finally, to recognize a few new variants of
known matrix pencil methods.
\end{itemize}

For given matrices $A_i,B_i,C_i\in\CC^{n\times n}$, $i=1,2$, the (linear)
two-parameter eigenvalue problem is concerned with finding
$\lambda,\mu\in\CC$ and $x,y\in\CC^{n}\backslash\{0\}$ such that
\begin{gather}
  \label{eq:2param} 
 \left\{
  \begin{array}{lll}
    A_1x&=\lambda B_1x&+\mu C_1x, \\
    A_2y&=\lambda B_2y&+\mu C_2y.
  \end{array}
\right.
\end{gather}
There is a close connection between linear two-parameter problems and
two coupled generalized eigenvalue problems involving Kronecker products;
see \cite{Atkinson:1972:MULTIPARAMETER} and Section~\ref{sect:2param}
for further details.

In this paper, we will consider
{\em polynomial two-parameter eigenvalue problems} and show that
there are associated (one-parameter) quadratic eigenvalue problems
which are very relevant for critical delays of DDEs.
Note that the use of {\em multivariate polynomials}, which are closely
related to multiparameter eigenvalue problems, is not new in the field
of stability 
 of DDEs. Multivariate polynomials are used in, e.g.,
\cite{Kamen:1982:LINCOM}, \cite{Kamen:1980:TWOVARIABLE}, and
\cite{Hale:1985:MULTIVAR} with applications in
\cite{Chiasson:1988:METHOD}; see also the summaries in the standard
references \cite[Section~4.1.2]{Niculescu:2001:DELAY} and 
\cite[Section 4.6]{Gu:2003:STABILITY}. 
In this work, we discuss a new natural way to interpret  
matrix pencil methods in the context of two-parameter eigenvalue
problems.

The results of this  work are ordered by increasing generality. The
idea of an alternative interpretation of matrix pencil methods
using polynomial two-parameter eigenvalue problems is first illustrated
in Section~\ref{sect:illustration}. In Section~\ref{sect:2param}
we give connections between certain polynomial two-parameter
eigenvalue problems and associated quadratic and polynomial eigenvalue
problem. These links are used to derive the polynomial
eigenvalue problem occurring in matrix pencil methods for more general
types of DDEs in Section~\ref{sect:gen}.
After stating some new variants of matrix pencil methods in
Section~\ref{sect:new}, we end with some conclusions and an outlook in
Section~\ref{sec:concl}.

\section{DDEs with a single delay}
\label{sect:illustration}
An important aspect of this work is a further understanding of
matrix pencil methods. The derivation for the most general type of DDE
is somewhat technical and contain expressions difficult to interpret by
inspection. Therefore, to ease the 
presentation, it is worthwhile to first illustrate the general
ideas of the theory by considering retarded DDEs with a single delay. 

In this section we derive a polynomial two-parameter eigenvalue problem 
corresponding to purely imaginary eigenvalues of a DDE, and apply a result
that will be proved in Theorem~\ref{thm:2paramcrossterm} in Section~\ref{sect:2param}
to identify that the eigenvalue problems are the ones that occur in the matrix
pencil methods proposed in \cite{Chen:1995:DELAYMARGINS} and \cite{Louisell:2001:IMAG}. 

First, we introduce the following (usual)
notations: $\sigma(A)$ and $\sigma(A,B)$ denote the spectrum of
a matrix $A$ and matrix pencil $(A,B)$, respectively;
$I$ denotes the identity matrix,
$\otimes$ the Kronecker product and
$\oplus$ the Kronecker sum (i.e., $A\oplus B=A\otimes I+I\otimes B$).
If a DDE is stable for $\tau = 0$, then $\tau_*$ denotes the {\em delay margin},
i.e., the smallest delay $\tau$ for which the DDE is no longer stable. 

Consider the DDE
\begin{equation}
 B_0\dot{x}(t)=A_0x(t)+A_1x(t-\tau),
 \label{dde1}
\end{equation}
where $A_0,A_1,B_0\in\CC^{n\times n}$. We will rederive the eigenvalue problems
that arise in the following two matrix pencil results.
We hereby note that matrix pencil methods are generally stated in various
degrees of generality, for various types of DDEs.
Theorem~\ref{thm:louisell} below is for the slightly different
setting of neutral DDEs and $B_0=I$. 
Theorem~\ref{thm:chen} as stated here is a restriction of the original result
\cite{Chen:1995:DELAYMARGINS} to single delays.
We postpone the discussion of the more general result in
\cite{Chen:1995:DELAYMARGINS} to Section 4.

\begin{thm}[Louisell \textrm{\cite[Theorem~3.1]{Louisell:2001:IMAG}}]
\label{thm:louisell}
Let $A_0,A_1,B_1\in\RR^{n\times n}$. Then all purely imaginary eigenvalues
 of the neutral DDE 
\begin{equation}
\dot{x}(t)+B_1\dot{x}(t-\tau)=A_0x(t)+A_1x(t-\tau)\label{eq:louisell}
\end{equation}
are zeros of 
\begin{equation}
\det((\lambda I-A_0)\kron(\lambda I+A_0)-
(\lambda B_1-A_1)\otimes (\lambda B_1+A_1))=0.\label{eq:louisell:quadpoly}
\end{equation}
\end{thm}

\begin{thm}[Chen, Gu, and Nett; special case of
\textrm{\cite[Theorem~3.1]{Chen:1995:DELAYMARGINS}}]\label{thm:chen}
  Suppose \eqref{dde0} is stable for $\tau=0$. Define 
\begin{equation}
   U:=
\mtxa{cc}{
I & \ 0 \\
0 & \ A_1\otimes I
}
\quad \textrm{and} \quad
V:=
\mtxa{cc}{
 0   & \ I   \\
-I\otimes A_1^T & \ -A_0\oplus A_0^T
}.
\end{equation}
If the delay margin $\tau_*$ is finite and
 nonzero, then $\tau_*=\min_{k} \frac{\alpha_k}{\omega_k}$
where $\alpha_k\in [0,2\pi]$, $\omega_k > 0$, and 
$e^{-i\alpha_k}\in \sigma(V,U)$ satisfies the relation $i\omega_k\in\sigma(A_0+A_1e^{-i\alpha_k})$.
\end{thm}

Theorem~\ref{thm:chen}
gives a formula for the delay
margin in terms of the solutions of the generalized
eigenvalue problem involving the pencil $(V,U)$, which represents a
linearization of the {\em quadratic eigenvalue problem} (QEP)
\begin{equation}
 \left(\mu^2(A_1\otimes I)+\mu \, (A_0\oplus A_0^T)+ I\otimes A_1^T\right)v = 0
\label{eq:chenpoly}
\end{equation}
for $\mu \in \CC$ and nonzero $v\in\CC^{n^2}$.
An exhaustive characterization of possible linearizations
was recently given in \cite{Mackey:2006:STRUCTURED}. 
Observe that the matrix polynomial \eqref{eq:louisell:quadpoly} in
Theorem~\ref{thm:louisell} also represents a quadratic eigenvalue problem;
a linearization was given in \cite{Louisell:2001:IMAG} as well.
(A linearization adapted to the quadratic eigenvalue problem
in the matrix pencil method in \cite{Jarlebring:2008:JCAM} was 
recently proposed in \cite{Fassbender:2007:PCP}.)

Both of the matrix pencil methods (Theorem~\ref{thm:louisell}
and Theorem~\ref{thm:chen}) involve quadratic
eigenvalue problems \eqref{eq:louisell:quadpoly} and
\eqref{eq:chenpoly}. However, for instance from the original proofs
of these results, there is no obvious relation between these two approaches. 
We will develop a framework that can derive both quadratic
eigenvalue problems in a unifying manner which gains further insight
in the relations between the methods.

Consider the eigenvalue problem associated with \eqref{dde1}
\begin{equation}
  \lambda B_0x=(A_0+A_1e^{-\lambda \tau}) \, x,\label{eq:chareqvect}
\end{equation}
for nonzero $x\in\CC^n$. We are  interested in the case
where there is a purely imaginary eigenvalue, say $\lambda=i\omega$. We
denote $\mu=e^{-\lambda\tau}$. Under the assumption that the eigenvalue
is imaginary, i.e., $\lambda=i\omega$, we have 
$\cnj{\lambda}=-\lambda$ and $\cnj{\mu}=\mu^{-1}$. This yields 
\begin{equation}
  -\lambda \cnj{B}_0y=(\cnj{A}_0+\mu^{-1}\cnj{A}_1) \, y,\label{eq:chareqcnj}
\end{equation}
where $y=\cnj{x}$. Hence, multiplying \eqref{eq:chareqcnj} by $\mu$ and
rearranging the terms, we have
\begin{gather}
  \label{eq:quad2param} 
 \left\{
  \begin{array}{lll}
    A_0x      &=\phantom{-}\lambda B_0x &-\mu A_1x, \\
    \cnj{A}_1y&=-\lambda\mu \cnj{B}_0y&-\mu\cnj{A}_0y.
  \end{array}
\right.
\end{gather}

Now first, for given $A_i$, $B_i$, $C_i$, $D_i$, $E_i$, and $F_i \in \CC^{n \times n}$,
$i=1,2$, consider the following {\em quadratic two-parameter eigenvalue problem}
\[
\left\{
\begin{array}{lll}
(A_1 + \lambda B_1 + \mu C_1 + \lambda^2 D_1 + \lambda \mu E_1 + \mu^2 F_1) \, x & = & 0, \\
(A_2 + \lambda B_2 + \mu C_2 + \lambda^2 D_2 + \lambda \mu E_2 + \mu^2 F_2) \, y & = & 0,
\end{array}
\right.
\]
where the assignment is to compute one or more tuples $(\lambda, \mu, x, y)$ with
nonzero $x$ and $y$. As for the linear two-parameter eigenvalue problem we will call
$(\lambda, \mu)$ an eigenvalues and $x \otimes y$ an eigenvector.
We see that \eqref{eq:quad2param} is special case of this general quadratic
two-parameter eigenvalue problem with just one nonlinear term
and one additional vanishing matrix.
As an implication of Theorem~\ref{thm:2paramcrossterm} in the next section,
we will show that the following two (one-parameter) quadratic eigenvalue
problems are associated with \eqref{eq:quad2param}:
\begin{equation}
  \left[\lambda^2(B_0\otimes \cnj{B}_0)
+\lambda(B_0\otimes \cnj{A}_0-A_0\otimes \cnj{B}_0)
+(A_1\otimes \cnj{A}_1-A_0\otimes \cnj{A}_0)
\right](x\otimes y)=0\label{eq:singlequadlambda}
\end{equation}
and
\begin{equation}
\left[\mu^2(A_1\otimes \cnj{B}_0
+\mu \, (A_0\otimes B_0+B_0\otimes \cnj{A}_0)
+(B_0\otimes \cnj{A}_1))
\right](x\otimes y)=0. \label{eq:singlequadmu}
\end{equation}

Using these QEPs, we can now rederive Theorems~\ref{thm:louisell}
and Theorem~\ref{thm:chen} as follows.
Although Theorem~\ref{thm:louisell} applies to the wider class of neutral DDEs,
we can restrict it to the class of retarded DDEs by setting
$B_1=0$. Then taking $B_0=I$ in the quadratic eigenvalue problem
\eqref{eq:singlequadlambda} exactly renders \eqref{eq:louisell:quadpoly}
in Theorem~\ref{thm:louisell}, under the assumption that the matrices are real.

Similarly, \eqref{eq:singlequadmu}
corresponds to the quadratic eigenvalue problem \eqref{eq:chenpoly} in
Theorem~\ref{thm:chen}; note that \eqref{eq:singlequadmu} gives
\eqref{eq:chenpoly} if we replace conjugation with the conjugate transpose
as follows. Instead of \eqref{eq:chareqcnj} as the conjugate of
\eqref{eq:chareqvect}, we can also take the conjugate transpose
of \eqref{eq:chareqvect}.
(In Section~\ref{sect:new}, we will exploit similar techniques to
derive new matrix pencil methods.)
The resulting quadratic two-parameter eigenvalue problem is
\begin{gather}
  \label{eq:quad2paramtrans} 
 \left\{
  \begin{array}{lll}
    A_0x  &=\phantom{-}\lambda B_0x  &-\mu A_1x, \\
    A_1^*y&=-\lambda\mu B_0^*y&-\mu A_0^*y,
  \end{array}
\right.
\end{gather}
where now $y$ is the left eigenvector of \eqref{eq:chareqvect}. The second
equation in \eqref{eq:quad2paramtrans} is the transpose of
\eqref{eq:quad2param}. Application of Theorem~\ref{thm:2paramcrossterm}
to \eqref{eq:quad2paramtrans} yields \eqref{eq:chenpoly}.

In the next section we will prove the theorem that we need
for \eqref{eq:singlequadlambda} and \eqref{eq:singlequadmu}
and also more general results.

\section{Quadratic and polynomial two-parameter eigenproblems
and associated quadratic and polynomial one-parameter eigenproblems}
\label{sect:2param}
First, we will review some facts for the (linear) two-parameter eigenvalue
problem \eqref{eq:2param}, see also \cite{Atkinson:1972:MULTIPARAMETER}.
Define the matrix determinants
\begin{eqnarray*}
    \Delta_0&=&B_1\otimes C_2-C_1\otimes B_2, \\[1mm]
    \Delta_1&=&A_1\otimes C_2-C_1\otimes A_2, \\[1mm]
    \Delta_2&=&B_1\otimes A_2-A_1\otimes B_2,
\end{eqnarray*}
where $\Delta_i\in\CC^{n^2\times n^2}$, $i=0,1,2$.
Associated with \eqref{eq:2param} are two (decoupled)
generalized eigenvalue problems (GEPs)
\begin{gather}
 \begin{array}{lll}
    \Delta_1 z&=&\lambda \Delta_0z, \\
    \Delta_2 z&=&\mu \Delta_0z,
 \end{array}
\end{gather}
where $z = x \otimes y$.
(In fact, these GEPs are equivalent with \eqref{eq:2param} if
$\Delta_0$ is nonsingular; see \cite{Atkinson:1972:MULTIPARAMETER}).
As there are two generalized eigenvalue problems which correspond to the
linear two-parameter eigenvalue problem \eqref{eq:2param}, we will see
that there are two quadratic (one-parameter) eigenvalue problems which
correspond to the quadratic two-parameter eigenvalue problem. We will see in
the derivations of the matrix pencil methods that some methods
correspond to one form and some to the other.

To be able to handle two wider classes of DDEs, we will
prove two results that deal with generalizations of
problem \eqref{eq:quad2param};
we will make use of these theorems in the derivation of matrix pencil
methods in Section~\ref{sect:gen}.

The first generalization of \eqref{eq:quad2param}, which we will
use for neutral DDEs, involves an additional cross term $\lambda\mu$:
\begin{gather}
  \label{eq:2paramnl} 
 \left\{
  \begin{array}{llll}
    A_1x&=\lambda B_1x&+\mu C_1x&+\lambda\mu D_1x,\\
    A_2y&=\lambda B_2y&+\mu C_2y&+\lambda\mu D_2y.
  \end{array}
\right.
\end{gather} 

\begin{thm}\label{thm:2paramcrossterm}
If $(\lambda,\mu)$ is a solution of \eqref{eq:2paramnl} with
corresponding eigenvector $(x,y)$ then:
\begin{itemize}
\item[1.] $\lambda$ is an eigenvalue with corresponding eigenvector
$x\otimes y$ of the QEP
\begin{multline}
\label{eq:quad2paramlambda}
\big[\lambda^2(D_1\otimes B_2-B_1\otimes D_2)
+\lambda(A_1\otimes D_2-D_1\otimes A_2+\\
-B_1\otimes C_2+C_1\otimes B_2)+ (A_1\otimes C_2-C_1\otimes A_2)
\big] \, (x\otimes y)=0.
\end{multline}
\item[2.] $\mu$ is an eigenvalue with corresponding eigenvector
$x\otimes y$ of the QEP
\begin{multline}
\big[
\mu^2(D_1\otimes C_2-C_1\otimes D_2)
+\mu \, (A_1\otimes D_2-D_1\otimes A_2+\\
-C_1\otimes B_2+B_1\otimes C_2)+ (A_1\otimes B_2-B_1\otimes A_2)
\big] \, (x\otimes y)=0.
\end{multline}
\end{itemize}
\end{thm}
\begin{proof}
 We show the first implication; the second follows by switching the roles of
 $\lambda$ and $\mu$; $B_1$ and $B_2$; and $C_1$ and $C_2$.

 Equation \eqref{eq:quad2paramlambda} holds because 
 \begin{multline*}
\lambda^2(D_1\otimes B_2-B_1\otimes D_2)(x\otimes y)\\[1mm]
\quad =\lambda \, (D_1\otimes (A_2-\mu C_2-\lambda\mu D_2) 
        -(A_1-\mu C_1-\lambda\mu D_1)\otimes D_2)(x\otimes y)\\[1mm]
\quad =\lambda \, (D_1\otimes (A_2-\mu C_2) 
        -(A_1-\mu C_1)\otimes D_2)(x\otimes y)\\[1mm]
\quad =(\lambda \, (D_1\otimes A_2- A_1\otimes D_2)
        +\lambda\mu \, (C_1\otimes D_2-D_1\otimes C_2))(x\otimes y)\\[1mm]
\quad =(\lambda \, (D_1\otimes A_2- A_1\otimes D_2)+
 (C_1\otimes (A_2-\lambda B_2)-(A_1-\lambda B_1)\otimes C_2
)(x\otimes y),
 \end{multline*}
where we used that
\begin{eqnarray*}
&& \lambda\mu \, (C_1\otimes D_2-D_1\otimes C_2)(x\otimes y)\\[1mm]
&& \quad =(C_1\otimes (A_2-\lambda B_2-\mu C_2)-(A_1-\lambda B_1-\mu C_1)\otimes C_2)(x\otimes y)\\[1mm]
&& \quad =(C_1\otimes (A_2-\lambda B_2)-(A_1-\lambda B_1)\otimes C_2)(x\otimes y).
\end{eqnarray*}
\end{proof}

To derive the matrix pencil methods for DDEs with {\em multiple commensurate delays}
(that is, multiple delays that are all integer multiples of of some delay value
$\tau$, i.e., $\tau_k=\tau n_k$ where $n_k\in\NN$)
we will need the following more general polynomial two-parameter eigenvalue problem
\begin{gather}
  \label{eq:2paramnl2} 
 \left\{
  \begin{array}{lll}
    A_1x&=\lambda \ds \sum_{k=0}^m\mu^kB_{1,k}x&+\ds \sum_{k=1}^m\mu^kC_{1,k}x,\\[4mm]
    A_2y&=\lambda \ds \sum_{k=0}^m\mu^kB_{2,k}y&+\ds \sum_{k=1}^m\mu^kC_{2,k}y\\
  \end{array}
\right.
\end{gather}
in the next section.

We now prove the following result that will prove useful.
Associated with this polynomial two-parameter eigenvalue problem
is the following polynomial eigenvalue problem (PEP) for $\mu$. 
\begin{thm}\label{thm:2parampoly}
If $(\lambda,\mu)$ is an eigenvalue of \eqref{eq:2paramnl2} with
 eigenvector $(x,y)$ then
\begin{eqnarray*}
&& \big[(A_1\otimes B_{2,0}-B_{1,0}\otimes A_2)\\
&& \qquad +\sum_{k=1}^m
\mu^k (A_1\otimes B_{2,k}-B_{1,k}\otimes A_2
-C_{1,k} \otimes B_{2,0}+B_{1,0}\otimes C_{2,k})\\
&& \qquad +\sum_{k=1,i=1}^m \mu^{k+i}(
 B_{1,k} \otimes C_{2,i}-C_{1,k}\otimes B_{2,i})
\big] \, (x\otimes y)=0.
\end{eqnarray*}
\end{thm}
\begin{proof}
One may check that
\begin{gather*}
  \begin{array}{lll}
    A_1x&=\lambda B_1x+&\mu C_1x+\lambda\mu D_1x,\\
    A_2y&=\lambda B_2y+&\mu C_2y+\lambda\mu D_2y
  \end{array}
\end{gather*} 
if we let $B_i=B_{i,0}$, $D_i=\sum_{k=1}^m \mu^{k-1}B_{i,k}$ and
 $C_i=\sum_{k=1}^m\mu^{k-1}C_{i,k}$ for $i=1,2$. Application of
 Theorem~\ref{thm:2paramcrossterm} yields that
\begin{eqnarray*}
0&=&\big[(A_1\otimes B_{2,0}-B_{1,0}\otimes A_2)\\
&&\quad +\mu \, \big(A_1\otimes \sum_{k=1}^m \mu^{k-1}B_{2,k}-\sum_{k=1}^m \mu^{k-1}B_{1,k}\otimes A_2\\
&&\quad -\sum_{k=1}^m\mu^{k-1}C_{1,k} \otimes B_{2,0}+B_{1,0}\otimes
 \sum_{k=1}^m\mu^{k-1}C_{2,k}\big)\\
&&\quad +\mu^2 \big(\sum_{k=1}^m
 \mu^{k-1}B_{1,k} \otimes \sum_{k=1}^m\mu^{k-1}C_{2,k}-\sum_{k=1}^m\mu^{k-1}C_{1,k}\otimes \sum_{k=1}^m \mu^{k-1}B_{2,k}\big)
\big] \, (x\otimes y)\\
&=&\big[(A_1\otimes B_{2,0}-B_{1,0}\otimes A_2)\\
&&\quad +\big(A_1\otimes \sum_{k=1}^m \mu^{k}B_{2,k}-\sum_{k=1}^m \mu^{k}B_{1,k}\otimes A_2\\
&&\quad -\sum_{k=1}^m\mu^{k}C_{1,k} \otimes B_{2,0}+B_{1,0}\otimes
 \sum_{k=1}^m\mu^{k}C_{2,k}\big)\\
&&\quad +\big(\sum_{k=1}^m
 \mu^{k}B_{1,k} \otimes \sum_{k=1}^m\mu^{k}C_{2,k}-\sum_{k=1}^m\mu^{k}C_{1,k}\otimes \sum_{k=1}^m \mu^{k}B_{2,k}\big)
\big] \, (x\otimes y),
\end{eqnarray*}
which completes the proof.
\end{proof}

\section{Generalizations for neutral systems and multiple delays}\label{sect:gen}
In Section~\ref{sect:illustration} we used the
quadratic two-parameter eigenvalue problem \eqref{eq:quad2param}
to derive the quadratic eigenvalue problems in Theorem~\ref{thm:louisell} and Theorem~\ref{thm:chen}. However, in Section~\ref{sect:illustration} we
limited ourselves to the setting of a single delay DDE.
The original formulations of Theorem~\ref{thm:louisell} \cite{Louisell:2001:IMAG}
and Theorem~\ref{thm:chen} \cite{Chen:1995:DELAYMARGINS}, were stated
for more general types of DDEs, which we will study in this section.
In particular, we discuss neutral systems in Section~\ref{sect:neutral}
and the DDEs with multiple commensurate DDEs in Section~\ref{sect:multiple}.



\subsection{Neutral DDEs}
\label{sect:neutral}

Consider the neutral DDE
\[
 B_0\dot{x}(t)+B_1\dot{x}(t-\tau)=A_0x(t)+A_1x(t-\tau),
\]
where $A_0,A_1,B_0,B_1\in\CC^{n\times n}$. The generality of 
Theorem~\ref{thm:2paramcrossterm} allows us to derive the matrix pencil
methods for this DDE in a similar way as in Section~\ref{sect:illustration}.
With $\lambda=i\omega$ and $\mu=e^{-i\tau\omega}$ we
note that the corresponding eigenvalue problem and its complex conjugate can be
expressed as
\begin{gather}
  \label{eq:quad2paramneut} 
 \left\{
  \begin{array}{llll}
    A_0x      &=\phantom{-}\lambda B_0x & +\lambda\mu B_1x &-\mu A_1x, \\
    \cnj{A}_1y&=-\lambda \cnj{B}_1y & -\lambda\mu \cnj{B}_0y&-\mu\cnj{A}_0y.
  \end{array}
\right.
\end{gather}
After applying Theorem~\ref{thm:2paramcrossterm} 
we derive that  
\begin{multline*}
\big[(-A_0\otimes \cnj{A}_0+A_1\otimes \cnj{A}_1)
+\lambda(-A_0\otimes \cnj{B}_0-B_1\otimes \cnj{A}_1\\
+B_0\otimes \cnj{A}_0+A_1\otimes \cnj{B}_1)+
\lambda^2(-B_1\otimes \cnj{B}_1+B_0\otimes \cnj{B}_0)
\big] \, (x \kron y)=0,  
\end{multline*}
and after rearranging the terms we get
\begin{equation}
 \left((\lambda B_0 - A_0 )\otimes (\lambda \cnj{B}_0 + \cnj{A}_0 ) - (\lambda B_1 - {A_1} )\otimes
 (\lambda \cnj{B}_1 + \cnj{A}_1 )\right)(x\otimes y) = 0.
\end{equation}
This is a slight generalization of the eigenvalue problem presented by
Louisell \cite{Louisell:2001:IMAG}, since in \cite{Louisell:2001:IMAG} it is 
assumed that $B_0=I$ and that the matrices are real.
Louisell, motivated by a connection with a certain differential equation of which
all purely imaginary eigenvalues coincide with purely imaginary 
eigenvalues of the DDE, suggests that \eqref{eq:louisell} can be
determined by solutions of the generalized eigenvalue problem
\[
 \lambda 
 \mtxa{cc}{
     I\otimes I& \ B_1\otimes I\\ I\otimes B_1& \ I\otimes I
 }w=
 \mtxa{rr}{
    A_0\otimes I& \ A_1\otimes I\\ -I\otimes A_1 & \ -I\otimes A_0
 }w.
\]
However, we note that
this is just one possible linearization of \eqref{eq:louisell}; any
of the linearizations in \cite{Mackey:2006:VECT,Mackey:2006:STRUCTURED}
might be considered. Moreover, there also exist numerical methods for quadratic
eigenvalue problems that try to avoid linearization; see 
\cite{Tisseur:2001:QUADRATIC} for an overview.

The second resulting quadratic eigenvalue problem from applying
Theorem~\ref{thm:2paramcrossterm} to \eqref{eq:quad2paramneut} reads
\begin{multline*}
\big[\mu^2(B_1\otimes \cnj{A}_0+A_1\otimes \cnj{B}_0)
+\mu \, (A_0\otimes \cnj{B}_0+B_1\otimes \cnj{A}_1+\\
+A_1\otimes \cnj{B}_1+B_0\otimes \cnj{A}_0)
+(A_0\otimes \cnj{B}_1+B_0\otimes \cnj{A}_1)
\big] \, (x\otimes y)=0.
\end{multline*}
At this point we note that we can interchange all left and right 
operators in the Kronecker products to get a special case of the
result in \cite{Jarlebring:2007:ECC} (where multiple delays are considered).

To determine a relation with a result by
Fu, Niculescu, and Chen \cite{Fu:2006:NEUTRAL} we note that
similarly to the derivation of \eqref{eq:quad2paramneut} using $x$ and
$y=\cnj{x}$, we can also derive a quadratic two-parameter eigenvalue
problem involving $x$ and its corresponding left eigenvector $y$:
\begin{gather}
  \label{eq:quad2paramneut2} 
 \left\{
  \begin{array}{llll}
    A_0x  &=\phantom{-}\lambda B_0x&+\lambda\mu B_1x &-\mu A_1x, \\
    A_1^*y&=-\lambda B_1^*y&-\lambda\mu B_0^*y&-\mu A_0^*y.
  \end{array}
\right.
\end{gather}
Application of Theorem~\ref{thm:2paramcrossterm} yields
\begin{multline*}
\big[(A_0\otimes B_1^*+B_0\otimes A_1^*)
+\mu \, (A_0\otimes B_0^*+B_1\otimes A_1^*+\\
+A_1\otimes B_1^*+B_0\otimes A_0^*)+
\mu^2(B_1\otimes A_0^*+A_1\otimes B_0^*)
\big] \, (x\otimes y)=0.
\end{multline*}
This is a special case of the method in \cite{Fu:2006:NEUTRAL} which
applies to DDEs with multiple commensurate delays, which we be the topic
of the next subsection)\footnote{Note that $B_1$ is defined with an opposite sign in \cite{Fu:2006:NEUTRAL}.}.

\subsection{Multiple delays} \label{sect:multiple}
We now consider the case of DDEs with multiple, say $m>1$, delays.
In the literature there are essentially two ways to handle this situation.
Either the curves or surfaces corresponding to the critical
delays are parameterized using $m-1$ free variables, as is done in for instance
\cite{Jarlebring:2008:JCAM, Ergenc:2007:EXTENDED, Jarlebring:2007:ECC}. In
other approaches, e.g., \cite{Chen:1995:DELAYMARGINS, Fu:2006:NEUTRAL},
it assumed that the delays are commensurate.


Here, we will focus on the case of multiple commensurate delays as the
parameterizations do not yield stability information after the solution
of one eigenvalue problem.

%

Consider the DDE with commensurate delays
\[
  B_0\dot{x}(t)=\sum_{k=0}^mA_kx(t-\tau k).
\]
The associated eigenvalue problem is
\[
   \left(\sum_{k=0}^me^{-\tau k\lambda}A_k-\lambda B_0\right)v=0.
\]
As in the previous section we substitute $\lambda=i\omega$ and
$\mu=e^{-i\tau\omega}$ and consider the complex conjugate of the eigenvalue problem.
After rearrangement of the terms and sums we have
\begin{gather}
  \label{eq:poly2paramcrit} 
 \left\{
  \begin{array}{lll}
    -\cnj{A}_mu    &=\lambda\mu^m \cnj{B}_0u&+\ds \sum_{k=1}^m\mu^k \cnj{A}_{m-k}u, \\[4mm]
    \phantom{-}A_0v&=\lambda B_0v&-\ds \sum_{k=1}^m\mu^k A_kv.
  \end{array}
\right.
\end{gather} 
This is of the same form as the polynomial two-parameter eigenvalue
problem in \eqref{eq:2paramnl2} with 
$A_1=-A_m$, $B_{1,m}=B_0$, $B_{1,k}=0$, 
$k=0,\ldots m-1$, $C_{1,k}=A_{m-k}$, $k=1,\ldots,m$,
$A_2=A_0$, $B_{2,0}=B_0$, $B_{2,k}=0$, $k=1,\ldots m$,  $C_{2,k}=-A_k$, $k=1,\ldots,m$.

%

Theorem~\ref{thm:2parampoly} and several manipulations of the sums yield
\begin{eqnarray*}
0&=&\big[-A_m\otimes B_0+
\mu^m (-B_0\otimes A_0)+
\sum_{k=1}^m
\mu^k (
-A_{m-k} \otimes B_0)\\
&&\qquad \qquad \qquad \qquad \qquad \qquad \qquad \qquad +\sum_{i=1}^m
\mu^{m+i}(
 -B_0 \otimes A_i)
\big] \, (v\otimes u) \\
&=&\big[
-\sum_{k=0}^{m}
\mu^{m-k} (
A_{k} \otimes B_0)-\sum_{i=0}^m
\mu^{m+i}(
 B_0 \otimes A_i)
\big] \, (v\otimes u),
\end{eqnarray*}

which is the eigenvalue problem in derived in \cite{Jarlebring:2008:JCAM}.

%
%
%
%
%
%
%

As in Section~\ref{sect:illustration} and the neutral case
in the previous subsection, we may consider the conjugate transpose instead of the
transpose. This yields the polynomial eigenvalue
problem in \cite[Thm.~3.1]{Chen:1995:DELAYMARGINS}.

Finally, the most general result is for neutral commensurate DDEs. We
show that the eigenvalue problem in Fu, Niculescu, and Chen \cite{Fu:2006:NEUTRAL} also
is a polynomial eigenvalue problem that is connected with a polynomial two-parameter eigenvalue problem.
Although the analysis is similar as for the previous cases, this general
case involves more technicalities and more involved
expressions. Consider the polynomial two-parameter eigenvalue problem
corresponding to the neutral commensurate DDE
\[
\sum_{k=0}^mB_k\dot{x}(t-\tau k)
  =\sum_{k=0}^mA_kx(t-\tau k),
\]
i.e.,
\begin{equation}
\left(A_0-\sum_{k=0}^m \lambda\mu^kB_k+\sum_{k=1}^m\mu^kA_k\right)v=0.\label{eq:CD:mp:commchareq1}
\end{equation}
The complex conjugate transpose is 
\begin{equation}
  \left(\mu^mA_0^*+\sum_{k=0}^m \lambda\mu^{m-k}B_k^*+\sum_{k=1}^m \mu^{m-k}A_k^*\right)u=0.\label{eq:CD:mp:commchareq2}
\end{equation}

We can now combine \eqref{eq:CD:mp:commchareq1} and
\eqref{eq:CD:mp:commchareq2} into a polynomial two-parameter eigenvalue
problem 
\begin{gather}
  \label{eq:CD:poly2paramcrit} 
 \left\{
  \begin{array}{lll}
    \phantom{-}A_0v&=\ds \lambda \sum_{k=0}^m \mu^kB_kv&\ds -\sum_{k=1}^m \mu^k A_kv, \\[4mm]
    -A_m^*u&=\ds \lambda\sum_{k=0}^m \mu^k{B_{m-k}^*} u&\ds +\sum_{k=1}^m \mu^k A_{m-k}^*u.
  \end{array}
\right.
\end{gather} 
This corresponds to \eqref{eq:2paramnl2} with
$A_1=A_0$, $B_{1,k}=B_k$, 
$k=0,\ldots m$, $C_{1,k}=-A_{k}$, $k=1,\ldots,m$,
$A_2=-A_m^*$, $B_{2,k}=B_{m-k}^*$, $k=1,\ldots m$,  $C_{2,k}=A_{m-k}^*$,
$k=1,\ldots,m$.  Theorem~\ref{thm:2parampoly} yields

\begin{eqnarray*}
&& \big[(A_0\otimes B_{m-0}^*+B_0\otimes A_m^*)\\
&& \quad +\sum_{k=1}^m
\mu^k (A_0\otimes B_{m-k}^*+B_k\otimes A_m^*
+A_k\otimes B_{m-0}^*+B_0\otimes A_{m-k}^*)\\
&& \quad +\sum_{k=1,i=1}^m
\mu^{k+i}(
 B_k \otimes A_{m-i}^*+A_k\otimes B_{m-i}^*)
\big] \, (v\otimes u)=0.
\end{eqnarray*}
We note that with some effort it can be verified that the matrix coefficients
$Q_k$ in \cite[Thm.~2]{Fu:2006:NEUTRAL} are exactly the matrix coefficients
that occur in this polynomial eigenvalue problem.

\section{New variants of matrix pencil methods}
\label{sect:new}
In this section, we introduce some new matrix pencil methods, which are
variants of existing approaches. 
For ease of presentation, we will state the results for neutral DDEs
with one delay, but all methods can be generalized for DDEs with multiple
commensurate delays.

Moreover, we will only mention the relevant quadratic two-parameter
eigenvalue problems (which will be polynomial two-parameter eigenvalue problems
for DDEs with multiple commensurate delays);
as we have seen before, every such two-parameter eigenvalue problem
has two associated (one-parameter) eigenvalue problems,
one for $\lambda$ and one for $\mu$ giving two possible resulting
matrix pencil methods.
(For DDEs with multiple commensurate delays there is just one
associated polynomial eigenproblem, for $\mu$.)

The quadratic two-parameter eigenvalue problem for the
the neutral single-delay DDE given in \eqref{eq:quad2paramneut} is just
one of several possible quadratic two-parameter eigenvalue problems.
We can get the following expressions by transposing
none, one, or both equations:
\begin{itemize}
\item[a)] \eqref{eq:quad2paramneut};
\item[b)] \eqref{eq:quad2paramneut} but with the first equation transposed:
\[
 \left\{
  \begin{array}{llll}
    A_0^Tx    &=\phantom{-}\lambda B_0^Tx&+\lambda\mu B_1^Tx&-\mu A_1^Tx,\\
    \cnj{A}_1y&=-\lambda \cnj{B}_1y&-\lambda\mu \cnj{B}_0y  &-\mu\cnj{A}_0y;
  \end{array}
\right.
\]
\item[c)] \eqref{eq:quad2paramneut} but with the second equation transposed:
\begin{gather}
 \label{eq:trans2}
 \left\{
  \begin{array}{llll}
    A_0x  &=\phantom{-}\lambda B_0x&+\lambda\mu B_1x&-\mu A_1x,\\
    A_1^*y&=-\lambda B_1^*y&-\lambda\mu B_0^*y&-\mu A_0^*y;
  \end{array}
\right.
\end{gather}
\item[d)] and \eqref{eq:quad2paramneut} but with both equations transposed:
\[
 \left\{
  \begin{array}{llll}
    A_0^Tx&=\phantom{-}\lambda B_0^Tx&+\lambda\mu B_1^Tx&-\mu A_1^Tx,\\
    A_1^*y&=-\lambda B_1^*y&-\lambda\mu B_0^*y&-\mu A_0^*y.
  \end{array}
\right.
\]
\end{itemize}
Applying any of the two parts of Theorem~\ref{thm:2paramcrossterm}
yields an associated GEP corresponding to a matrix pencil method. As an
additional permutation, the order of the two equations in a
two-parameter eigenvalue problem does not influence the problem and can
be interchanged to yield yet other variants. Hence, in total for the neutral
single delay DDE we find $4\cdot 2\cdot 2=16$ matrix pencil variants.
(For DDEs with multiple delays there is one associated polynomial eigenvalue
problem resulting in 8 variants.)

The methods known in the literature correspond to the following:
\begin{itemize}
\item \cite{Chen:1995:DELAYMARGINS} and \cite{Fu:2006:NEUTRAL}:
\eqref{eq:trans2} and the $\lambda$-part of Theorem~\ref{thm:2paramcrossterm};
\item \cite{Louisell:2001:IMAG}: \eqref{eq:quad2paramneut} and
the $\mu$-part of Theorem~\ref{thm:2paramcrossterm};
\item \cite{Jarlebring:2008:JCAM}, \cite{Jarlebring:2007:ECC},
and \cite{Ergenc:2007:EXTENDED}:
\eqref{eq:quad2paramneut} and the $\lambda$-part of Theorem~\ref{thm:2paramcrossterm}.
\end{itemize}

Finally, we stress that the above list, which contains many
new variants, is more than just an
theoretical encyclopedic description of all possible options.
Depending on the given matrices, the structure
and sparsity patterns of the Kronecker products may differ
which may imply that for certain applications some methods may
be more favorable than others.

\section{Conclusions and outlook}
\label{sec:concl}
We have recognized new types of eigenvalue problems:
quadratic and polynomial two-parameter eigenvalue problems.
Using these problems as a unifying framework,
we have derived associated (one-parameter) quadratic or polynomial
eigenvalue problems that are at the heart of many matrix pencil methods
that are used to analyze asymptotic stability of DDEs. 
This unifying way to derive the matrix pencils in the matrix pencil methods
provides further understanding of these methods
and makes it easier to compare various approaches.
Moreover, we have proposed several new variations on known
matrix pencil methods.

Furthermore, we expect that the recognized framework of
quadratic and polynomial two-parameter eigenvalue problem may lead to
a considerable amount of new research.
First, we want to stress that it has been outside of the scope of this paper
to study theoretical and practical properties of these new types
of eigenvalue problems.
There are many interesting aspects that need further investigation,
such as how to carry over the concept of linearization (as is common
practice for QEPs and PEPs) to these problems.

Second, the matrix pencils constructed by Kronecker products that occur
in the matrix pencil methods are of large dimension by nature, even for
medium-sized problems, which may make efficient computation of
eigenvalues and stability of DDEs very
challenging. We believe that the key to a successful computational approach lies
in a direct attack of the polynomial two-parameter eigenvalue problem,
instead of the corresponding matrix pencils, in the same spirit
as, for instance, \cite{Hochstenbach:2005:TWOPARAMETER} for
the linear two-parameter eigenproblem. 
We leave both of these topics for future work.
Third, we note that an upcoming work by Muhi\v{c} and Plestenjak will
examine relations between quadratic two-parameter eigenvalue problems 
and singular linear two-parameter eigenvalue problems \cite{MPl08}.

\bigskip\noindent{\bf Acknowledgment:}
The authors would like to dedicate this work to Prof.~Henk van der Vorst,
who has always pointed us to new interesting and challenging problems
and has provided inspirational ideas to deal with them.

\bibliographystyle{elsart-num-sort}
\bibliography{fullbib}

\end{document}